\documentclass[11pt,oneside,reqno]{amsart}
 \usepackage{amsmath}
\usepackage{amssymb}
\usepackage{color}
\usepackage{cite}
\usepackage{enumerate}
\usepackage{fancyhdr}
\usepackage{graphicx}
\usepackage{bm}
\usepackage{soul,xcolor}

\usepackage[margin=3.5cm]{geometry}

\usepackage{subfigure}
\usepackage{algorithm}
\usepackage{algorithmic}
\usepackage{float}
\usepackage{lipsum}
\makeatletter
\newenvironment{breakablealgorithm}
{
	\begin{center}
		\refstepcounter{algorithm}
		\hrule height.8pt depth0pt \kern2pt
		\renewcommand{\caption}[2][\relax]{
			{\raggedright\textbf{\ALG@name~\thealgorithm} ##2\par}%
			\ifx\relax##1\relax 
			\addcontentsline{loa}{algorithm}{\protect\numberline{\thealgorithm}##2}%
			\else 
			\addcontentsline{loa}{algorithm}{\protect\numberline{\thealgorithm}##1}%
			\fi
			\kern2pt\hrule\kern2pt
		}
	}{
		\kern2pt\hrule\relax
	\end{center}
}
\makeatother

\numberwithin{equation}{section}

\usepackage{hyperref}\usepackage{hyperref}
\def\span{\textrm{ span }}

\def\C{ {\mathbb C} }
\def\R{ {\mathbb R} }

\def\Z{ {\mathbb Z} }

\def\T{ {\mathbb T} }

\DeclareMathOperator{\Tr}{Tr}

\newtheorem{theorem}{Theorem}[section]
\newtheorem{lemma}[theorem]{Lemma}
\newtheorem{proposition}[theorem]{Proposition}

\newtheorem{corollary}[theorem]{Corollary}

\newtheorem{example}[theorem]{Example}

\def\C{ {\mathbb C} }

\def\R{ {\mathbb R} }
\def\Z{ {\mathbb Z} }

\def\span{{\rm{span}}}
\numberwithin{equation}{section}

\title{Conjugate phase retrieval in a complex shift-invariant space }
\author{ Yang Chen and  Yanan Wang}
\address{Chen and Wang: Key Laboratory of Computing and Stochastic Mathematics (Ministry of Education), School of Mathematics and Statistics,
 Hunan Normal University, Changsha, Hunan 410081, P. R. China}
 \email{Chen: ychenmath@hunnu.edu.cn, Wang: ynwang163@163.com}

 \thanks{This project is  supported by  National Natural
Science Foundation of China (11901192, 12171490).}
\date{}

\begin{document}
	\date{}
	\pagestyle{plain}
	
	\date{}
	\maketitle
\begin{abstract}
The  conjugate phase retrieval problem concerns the determination of a complex-valued function, up to a unimodular constant and conjugation, from its magnitude observations. It can also be considered as a conjugate phaseless sampling and reconstruction problem in an infinite dimensional space. 
 In this paper, we first characterize the conjugate phase retrieval  from the point evaluations in  a  complex shift-invariant space $\mathcal S(\phi)$, where the generator $\phi$ is a compactly supported real-valued function.
If the generator $\phi$ has some spanning property, we also show that a conjugate phase retrievable function in $\mathcal S(\phi)$  can be reconstructed from its phaseless samples taken on a discrete set with finite sampling density. 
With additional phaseless measurements on the function derivative, for  the B-spline generator $B_N$ of order $N\ge 3$ which does not have the spanning property,  we find  sets $\Gamma$ and $\Gamma'\subset (0,1)$ of cardinalities $2N-1$ and $2N-5$  respectively,
such that a conjugate phase retrievable function $f$ in the spline space $\mathcal B_N$
can be determined from its phaseless Hermite samples $|f(\gamma)|, \gamma\in\Gamma+\Z$,  and $|f'(\gamma)|, \gamma'\in\Gamma'+\Z$. 
An algorithm is  proposed for the conjugate phase retrieval of piecewise polynomials from the Hermite samples.  Our results  provide  illustrative examples of real conjugate phase retrievable frames for the complex finite dimensional space $\C^N$.
 \end{abstract}

 {\bf Keywords}  conjugate phase retrieval; complex shift-invariant space;  B-spline; Hermite sampling.

\section{Introduction}
 In this paper, we consider the problem of reconstructing a function $f$ in a complex shift-invariant space $\mathcal S(\phi)$, up to a unimodular constant and conjugation, from the phaseless samples $|f(x)|, x\in\Omega$, and $|f'(x)|, x\in\Omega'$,  where the generator $\phi$ is a real-valued continuous function with  compact support, the complex shift-invariant space 
  \begin{equation}\label{sis.def}
\mathcal S(\phi):=\left\{\sum_{k\in\Z}c_k\phi(x-k), c_k\in\C {\rm \ for\ all\ } k\in \Z\right\},
\end{equation}  and  the sets $\Omega,\Omega'$ are either the whole real line $\R$ or its discrete subsets. The above problem is an  infinite-dimensional conjugate phase retrieval problem, which  determines a  function in a complex conjugate invariant linear space $\mathcal C$ (i.e., $\overline{f}\in \mathcal C$ if $f\in \mathcal C$), up to a unimodular constant and conjugation, from its phaseless measurements. It is a weak form of phase retrieval in the complex setting and has been discussed for the complex vectors from real frames and the continuous-time signals in a complex Paley-Wiener space or more general complex conjugate invariant linear spaces from the linear observations\cite{CCS22,Lai20,Lai21,McDonald04}.

The classical phase retrieval problem arises in optics\cite{F78,F82,HLO80,KS92,M90}. It aims to recover a function of interest from the magnitudes of its Fourier measurements. In the setting of frame theory, the authors of \cite{BCD06} introduced a mathematical formulation for the phase retrieval problem in 2006, which concerns the reconstruction of a vector in a real/complex Hilbert space, up to a global phase, from the magnitudes of its frame coefficients. Since then there are a lot of research devoted to this area. And it has been extended to the infinite-dimensional setting and generalized in other contexts. We refer to \cite{ADGT17,CCD16,CCPW16,CCSW20,CJS19,CS21,Edidin17,GSWX18,G20,HES16,IVW16,PYB14,QBP16,R21,T11,WX14,WX19,edidin22,JR22,ls20,SMS16,CS22,ADGY16,radu2016,RGrohs17,ABFM14,BBCE09,Bandeira14,candes2015,CSV12,cheng2020,DB22,MW15,xu}  for detailed discussions and recent developments on various aspects of this topic.

  Recently a new paradigm on phase retrieval of real vector-valued functions is proposed,
which is  to recover the vector-valued functions in a real linear space, up to an orthogonal transform, from the magnitudes of their linear functionals, and used in the velocity determination of a fleet of autonomous mobile robots and the unique configuration problem in Euclidean distance geometry \cite{CCS22,jackson2005,TBHTP19}. It is a generalized framework of  phase retrieval in the real setting and conjugate phase retrieval in the complex setting. Phase retrieval in the real setting can be seen as   phase retrieval of real scalar-valued functions.  By an isomorphism from $\C$ to $\R^2$, the conjugate phase retrieval problem of a complex-valued function  $f$ is essentially the phase retrieval problem of a $\R^2$-valued function  $(\Re f, \Im f)$.  The authors of \cite{CCS22} confirmed that conjugate phase retrieval in a complex conjugate invariant space $\mathcal C$ is closely related to  phase retrieval  in its real subspace $\Re \mathcal C=\{\Re f:=\frac{f+\bar f}{2}, f\in\mathcal C\}$. Specially they showed that a real frame  can do conjugate phase retrieval in $\C^2$ if and only if it allows phase retrieval in $\R^2$, and  found all the conjugate phase retrievable function in the complex shift-invariant space $V(h)$ generated by the hat function  $h(t)=\max(1-|t-1|, 0)$.

Shift-invariant space is widely used  in approximation theory,  wavelet analysis and signal processing, see\cite{AST05,AG01,deBoor92,mallatbook} and references therein. The representative examples are the Paley-Wiener space of bandlimited functions which  is a shift-invariant space generated by the function  ${\rm sinc}(t)=\frac{\sin \pi t}{\pi t}$, and the spline space 
\begin{eqnarray}\label{spline.sis.def}
\mathcal B_N\hskip-0.05in&:=\hskip-0.05in&\hskip-0.05in\left\{\sum_{k\in\Z}c_kB_N(x-k), c_k\in\C\right\}\nonumber\\
\hskip-0.05in&=\hskip-0.05in&\hskip-0.05in\left\{f\in \C^{N-2}, f|_{[j,j+1)}\ {\rm is\  a\ polynomial \ of \ degree}\ N-1 {\rm \ for\ all}\ j\in\Z \right\},
 \end{eqnarray} generated by the B-spline $B_N$ of order $N$, where $f|_{I}$ is the restriction of $f$ on the set $I$,  the B-spline $B_1$ is the characteristic function  on $[0,1)$ and the B-splines $B_N, N\ge 2$,  is defined by  $B_N=\int_0^1B_{N-1}(\cdot-t)dt$. Phase retrieval in a real  shift-invariant space $\mathcal S_\R(\phi)=\{\sum_{k\in\Z}a_k\phi(x-k), a_k\in\R\}$ has received a lot of attention \cite{CCSW20,CJS19,G20,R21,sun21}.
 In this paper, we consider the conjugate phase retrieval  in $\mathcal S(\phi)$ with a compactly supported real-valued generator $\phi$.
Following the methodology of local phase retrieval and sign propagation used for the phase retrieval in the real subspace $\mathcal S_\R(\phi)$ \cite{CCSW20},  in  Section \ref{cpr.sec} we characterize the conjugate phase retrievable functions from their point evaluations in $\mathcal S(\phi)$  via  local conjugate phase retrieval on the intervals $I_j=(j,j+1), j\in\Z,$ and propagating phase and conjugation among neighbouring intervals, see Theorem \ref{global.thm}.

(Conjugate) phase retrieval in a shift-invariant space  can be seen as a (conjugate) phaseless (Hermite) sampling and reconstruction problem \cite{CCSW20,CJS19,CS22,G20,ls20,R21,sun21}.
Sampling in a shift-invariant space is well studied as it is a realistic and suitable model for many applications \cite{AG01,AST05,GRS18,Sun10}.
   In Section \ref{sampling.sec}, we show that if the compactly supported real-valued generator $\phi$ with support length $L$ has some spanning property,
   the complex shift-invariant space $\mathcal S(\phi)$ is locally conjugate phase retrievable on the intervals $I_j=(j,j+1), j\in\Z$, and  all the conjugate phase retrievable functions in $\mathcal S(\phi)$ can be determined, up to a unimodular constant and conjugation, from its phaseless samples $|f(x)|,x\in\Gamma+\Z$, taken on the set $\Gamma+\Z$ with  sampling density $\frac{L(L+1)}{2}$, see Theorem \ref{sufficiency.thm}. The above spanning condition  and sampling density $\frac{L(L+1)}{2}$ is also necessary for the local conjugate phase retrievability of $\mathcal S(\phi)$ on the intervals $I_j, j\in\Z$, when the generator $\phi$ has support length $L=3$, see Corollary \ref{l3.cor}, Example \ref{phi1.ex} and  an counterexample of $\mathcal B_3$ in Example \ref{b3.ex}, and also an equivalent formulation for a conjugate phase retrievable frame of real vectors for $\C^3$ in \cite{Lai20}.

Hermite sampling is considered as interpolating by the samples from the function  and its first derivative, which is widely used in the reconstruction of bandlimited functions and Hermite splines\cite{FAUU20,JF56,M89}. The additional derivative measurements are introduced in the (conjugate) phase retrieval of analytic functions\cite{JR22,Lai21,McDonald04}. In Section \ref{hermite.sec}, we consider the conjugate phaseless sampling and reconstruction from Hermite samples in  $\mathcal S(\phi)$, see Theorem \ref{hermite.thm}. With the help of derivative samples, the conjugate phase retrievable function in $\mathcal B_3$ can be determined, up to a unimodular constant and conjugation, from its phaseless Hermite samples on some discrete sets.

  For any function $f$ in the spline space $\mathcal B_N$ as in \eqref{spline.sis.def} with $N\ge 3$, the nonzero restrictions $f|_{I_j}$ on the intervals $I_j=(j,j+1),j\in\Z,$ are polynomials of order $N-1$ and they may not be conjugate phase retrieval from its magnitudes $|f(x)|, x\in I_j$.
In Theorem \ref{hermite.spline.thm} of Section \ref{cps.b}, we show that $\mathcal B_N$ is local conjugate phase retrieval on the intervals $I_j,j\in\Z,$ from the  Hermite measurements, and a conjugate phase retrievable function $f\in \mathcal B_N$  can be determined, up to a unimodular constant and conjugation, from its phaseless Hermite samples $|f(\gamma)|, \gamma\in\Gamma+\Z$, and $|f'(\gamma')|, \gamma'\in\Gamma'+\Z$, taken on the discrete sets $\Gamma+\Z$ and $\Gamma'+\Z$ with sampling density $2N-1$ and $2N-5$ respectively. Based on the constructive proof, we propose an algorithm to implement the  conjugate phase retrieval of piecewise polynomials from the Hermite samples. Our results also provide illustrative examples for frames of $4N-6$ real vectors  that allows conjugate phase retrieval in $\C^N$ \cite{Lai20}.

\section{Global and local conjugate phase retrieval in a complex shift-invariant space}\label{cpr.sec}
For a compactly supported real-valued generator $\phi$ of the complex shift-invariant space $\mathcal S(\phi)$, let   \begin{equation*}\label{supportlength.def}
 L:=\min_{L_1, L_2\in \Z} \{L_2-L_1, \phi \ {\rm vanishes\ outside}  \ [L_1, L_2]\},\end{equation*}
 be its support length.  For the representative generator $B_N$, its support length is the same as the order $N\ge 1$.
 Without loss of generality,
\begin{equation}\label{supportlength.def2}
\phi(t)=0\ {\rm  for\ all} \ t\not\in [0, L],\end{equation}
otherwise replacing $\phi$ by $\phi(\cdot-L_0)$ for some $L_0\in \Z$. For an open set $A\subset \R$,
   a function $f\in \mathcal S(\phi)$ is  said to be {\it locally conjugate phase retrievable} on $A$ if for any $g\in \mathcal S(\phi)$ satisfying $|f(x)|=|g(x)|, x\in A$, there exists a unimodular constant $z\in \T$ such that $g=zf$ or $g=z\overline{f}$ on $A$, where $\mathbb T=\{z\in\C, |z|=1\}$,  and $\mathcal S(\phi)$ is said to be {\it locally  conjugate phase retrievable} on $A$ if all functions in $\mathcal S(\phi)$ are locally conjugate phase retrievable on $A$. If  a function $f\in \mathcal S(\phi)$ is locally conjugate phase retrievable on the whole line $\R$, 
   it is  said to be {\it globally conjugate phase retrievable} or {\it conjugate phase retrievable} for short.  Note that not all functions in $\mathcal S(\phi)$ are conjugate phase retrievable. For instance the functions   $z_1\phi(t)\pm z_2\phi(t-L)\in \mathcal S(\phi), z_1,z_2\in\T$,   have same magnitudes $|\phi(t)|+|\phi(t-L)|$, but they are not the same even up to a unimodular constant and conjugation.  In this section, we characterize all the conjugate phase retrievable  functions in $\mathcal S(\phi)$,
 in terms of local conjugate phase retrieval on  the intervals $I_j=(j,j+1), j\in\Z$, and propagating phase and conjugation  among neighbouring intervals, see  Theorem  \ref{global.thm}.

 Let us start from the simplest case that the generator has  support length $L=1$, a function $f\in \mathcal S(\phi)$ is conjugate phase retrievable if and only if there exists some integer $k_0\in\Z$ such that  \begin{equation}
 f(x)=c_{k_0}\phi(x-k_0) \  {\rm for\ some\ } c_{k_0}\in\C.
 \end{equation}
 For the case that the generator $\phi$ has support length $L=2$ and $\mathcal S(\phi)$ is locally conjugate phase retrievable on the intervals $I_j=(j,j+1), j\in\Z$,
 following a similar argument for the generator $h(t)=\max(1-|1-t|,0)$ in \cite{CCS22}, we have that a nonzero function $f(x)=\sum_{k\in\Z}c_k\phi(x-k)\in \mathcal S(\phi)$ is conjugate phase retrievable if and only if
 \begin{equation}  \label{complex.sis.pr.eq2}
c_k\ne 0 \  \ {\rm for \ all} \ K_-(f)-1< k<K_+(f)+1\end{equation}
and  there exists at most one $k_0\in (K_-(f)-1, K_+(f))$ such that
\begin{equation}\label{complex.sis.pr.eq3}
 \Im \big(c_{k_0}\overline{c_{k_0+1}}\big)\ne 0,
 \end{equation} where \begin{equation}\label{kpm.def}K_-(f):=\inf\{k, \ c_k\neq 0\} {\ \rm and }\  K_+(f):=\sup\{k, \  c_k\neq 0\}.\end{equation}
Therefore a nonzero function $f\in \mathcal S(\phi)$ is conjugate phase retrievable if and only if it satisfies \eqref{complex.sis.pr.eq2} and has the following decomposition
\begin{equation}\label{kfempty.decom}
f(x)=\xi_1\sum_{k\le k_0 } a_k\phi(x-k)+\xi_2\sum_{k\ge k_0+1} a_k\phi(x-k),
\end{equation} where $\xi_1,\xi_2\in \mathbb T$, $a_k=\overline{\xi_1}c_k\in\R, k\le k_0$, and $a_k=\overline{\xi_2}c_k\in\R, k\ge k_0+1$. For the case that  $f$ is real-valued up to a unimodular constant, the decomposition \eqref{kfempty.decom} holds for $\xi_1=\xi_2$, and hence $f$ is conjugate phase retrievable in $\mathcal S(\phi)$ if and only if \eqref{complex.sis.pr.eq2} holds  if and only if $\overline{\xi_1} f$ is phase retrievable in $\mathcal S_\R(\phi)$.

Now we  consider the case that the generator $\phi$ has support length
\begin{equation}\label{L.assump}
L\ge 3.\end{equation}
For any nonzero function $f\in \mathcal S(\phi)$, we can find  $\xi_1,\xi_2\in\mathbb T$ such that \begin{eqnarray}\label{f.decomp}
f(x)\hskip-0.1in&=\hskip-0.05in&\hskip-0.05in \xi_1\sum_{k=K_-(f)}^{ \kappa_-(f)}a_k\phi(x-k)+\sum_{k\in \mathcal K(f)}c_k\phi(x-k)+\xi_2\sum_{k=\kappa_+(f)}^{K_+(f)}a_k\phi(x-k), \nonumber\\
\hskip-0.05in&=:\hskip-0.05in&\hskip-0.05in \xi_1f_{\rm L}(x)+f_{\rm M}(x)+\xi_2f_{\rm R}(x),
\end{eqnarray}
where $K_\pm(f)$ is as in \eqref{kpm.def}, the function $f_{\rm L}(x)=\displaystyle{\sum_{k=K_-(f)}^{ \kappa_-(f)}}a_k\phi(x-k)$ is real-valued with \begin{equation}\label{coeff.eq.1}
 \kappa_-(f)\hskip-0.03in :=\hskip-0.03in\sup\{k, \Im (\overline{\xi_1}c_{l})=0{\ \rm for\ all\ }l\le k \} {\rm \ and\ } a_k=\overline{\xi_1}c_k\in \R {\rm\ for\ all\ }  k< \kappa_-(f)+1,\end{equation}  the function $f_{\rm R}(x)=\displaystyle{\sum_{k=\kappa_+(f)}^{K_+(f)}}a_k\phi(x-k)$ is real-valued with \begin{equation}\label{coeff.eq.inf}\kappa_+(f)=\inf\{k>\kappa_-, \Im (\overline{\xi_2}c_{l})= 0 {\ \rm for\ all\ }l\ge k\}
   \end{equation}
    and $\R\ni a_k=\left\{\begin{array}{cc}\overline{\xi_2}c_k & {\rm\ if\ }\kappa_+<K_+(f)+1\\
    0&{\rm else}, \end{array}
\right.$  and the function $f_{\rm M}(x)=\displaystyle{\sum_{k\in \mathcal K(f)}}c_k\phi(x-k)$ with
 \begin{equation}\label{kf.def}
{\mathcal K}(f):=\left\{\begin{array}{cc}(\kappa_-(f), \kappa_+(f))\cap \Z & {\rm \ if \ } \kappa_-+2\le\kappa_+<K_+(f)+1\\
\emptyset&{\rm else}.\end{array}
\right.
\end{equation}
Here we set $\inf \emptyset=+\infty$.
Using the above decomposition, we characterize the conjugate phase retrievable functions in $\mathcal S(\phi)$.

\begin{theorem}\label{global.thm}
{\rm  Let $\phi$ be a real-valued continuous function satisfying \eqref{supportlength.def2} and \eqref{L.assump}, and $\mathcal S(\phi)$ be locally conjugate phase retrievable on the intervals $I_j=(j,j+1), j\in\Z$.  Then a nonzero  function  $f(x)=\sum_{k\in\Z}c_k\phi(x-k)\in \mathcal S(\phi)$ with decomposition \eqref{f.decomp}-\eqref{kf.def} is  conjugate phase retrievable if and only if
\begin{itemize}
\item[i)] \begin{equation}\label{f.real.cond}
\sum_{l=0}^{L-2}|c_{k+l}|\ne0{\ \rm for \ all\ }k\in (K_{-}(f)-L+1,K_{+}(f)+1)
\end{equation} and
\item[ii)]
\begin{equation}\label{global.assump}
\sum_{k-L+2\le n_1<n_2\le k }|\Im(c_{n_1}\overline{c_{n_2}})|\ne 0 {\ \rm\ for\ all\ } \kappa_-+1\leq k\leq \kappa_++L-3,
\end{equation} provided that $\mathcal K(f)\ne \emptyset.$
 \end{itemize}
 }
\end{theorem}

 Observe that $\mathcal K(f)\ne \emptyset$ if and only if $-\infty<\kappa_-+2\le \kappa_+<\infty$ and \begin{equation}\label{coeff.eq.0}c_{\kappa_-+1}\ne 0{\rm\ and\ }c_{\kappa_+-1}\ne 0.\end{equation} For a nonzero function $f\in \mathcal S(\phi)$ with $\mathcal K(f)=\emptyset$, its decomposition \eqref{f.decomp}  becomes as in \eqref{kfempty.decom}, and  it is conjugate phase retrievable in  $ \mathcal S(\phi)$ if and only if  \eqref{f.real.cond}  holds. We remark  that the equality  in \eqref{global.assump} holds for  $L=2$ if and only if ${\mathcal K}(f)=\emptyset$,  then   \eqref{f.real.cond} and \eqref{global.assump} become \eqref{complex.sis.pr.eq2} and \eqref{complex.sis.pr.eq3} respectively and the decomposition of a conjugate phase retrievable function in \eqref{f.decomp} is as in \eqref{kfempty.decom}.
To prove  Theorem  \ref{global.thm}, we recall a characterization of phase retrieval in $\mathcal S_\R(\phi)$  and the relation between  phase retrieval in $\mathcal S_\R(\phi)$ and  conjugate phase retrieval in $\mathcal S(\phi)$.
 \begin{lemma}\label{real.pr.lem} {\rm[\cite{CCSW20}, Theorem 3.2]  Let $\phi$ be a real-valued continuous function satisfying \eqref{supportlength.def2} and \eqref{L.assump}, and the real shift-invariant space $\mathcal S_\R(\phi)$ is locally phase retrievable on the  intervals $I_j=(j,j+1), j\in\Z$. Then  a nonzero function $f(x)=\sum_{k\in \mathbb{Z}}d_k\phi(x-k)\in \mathcal S_{\R}(\phi)$ is phase retrievable if and only if
\begin{equation*}
\sum_{l=0}^{L-2}|d_{k+l}|^2\neq 0
\end{equation*}
for all $K_{-}(f)-L+1<k<K_{+}(f)+1$.}
\end{lemma}
\begin{lemma}\label{cpr.lem}{\rm[\cite{CCS22}, Theorem 2.3]  Let $\phi$ be a real-valued continuous function satisfying \eqref{supportlength.def2} and \eqref{L.assump}. If a function $f(x)=\sum_{k\in \mathbb{Z}}c_k\phi(x-k)\in \mathcal S(\phi)$ is  conjugate phase retrievable in $\mathcal S(\phi)$, then for all $a,b\in \mathbb{R}$,  the linear combinations $g(x)=\sum_{k\in \mathbb{Z}}(a\Re c_k+b\Im c_k)\phi(x-k)$ are phase retrievable in the real subspace $\mathcal S_\R(\phi)$.}
\end{lemma}

\begin{proof}[Proof of Theorem \ref{global.thm}] Set $K_\pm=K_\pm (f)$ and $ \kappa_\pm=\kappa_\pm(f)$, and write $f$ as in  \eqref{f.decomp}-\eqref{kf.def}.
$\Longrightarrow:$
i) If  the conjugate phase retrievable function $f\in \mathcal S(\phi)$ is real-valued,
\eqref{f.real.cond} follows immediately   by Lemmas \ref{real.pr.lem} and \ref{cpr.lem}.
Otherwise, by \eqref{coeff.eq.1}-\eqref{coeff.eq.inf}, we have  \begin{equation}\label{coeff.eq.2}
     K_-(\Re(\overline{\xi_1}f))=K_-{\rm\ and\ }K_+(\Re(\overline{\xi_2}f))=K_+.
     \end{equation}
 By Lemma \ref{cpr.lem},   the real-valued functions $\Re(\overline {\xi_1}f)$ and $\Re (\overline{\xi_2}f)$ are phase retrievable in $\mathcal S_\R(\phi)$.  This together with  \eqref{coeff.eq.2} and  Lemma \ref{real.pr.lem} implies that
 \begin{equation}\label{coeff.con.1}
 \sum_{l=0}^{L-2}|c_{k+l}|\ne 0 {\rm \ for\  all\ } K_--L+1<k\le K_+(\Re (\overline{\xi_1} f)){\rm\  and  \ }K_-(\Re (\overline{\xi_2} f))-L+1<k< K_++1.
   \end{equation}
   Observe that   for the case that $\Re (\xi_1\overline{\xi_2})\ne  0$, we have
     \begin{equation*}\label{coeff.eq.3}
     K_+(\Re(\overline{\xi_1}f))\ge\kappa_+-1\ge \kappa_-\ge K_-(\Re(\overline{\xi_2}f))-L+1.
     \end{equation*} Hence \eqref{f.real.cond} holds by \eqref{coeff.con.1}. Otherwise following a similar argument on the function  $\Re (e^{-\frac{\pi}{4}i}\overline{\xi_1} f)$ with
     \begin{equation*}
     K_-(\Re (e^{-\frac{\pi}{4}i}\overline{\xi_1} f))\le K_+(\Re(\overline{\xi_1}f)){\rm \  and \ }  K_+(\Re (e^{-\frac{\pi}{4}i}\overline{\xi_1} f))\ge K_-(\Re(\overline{\xi_2}f)),
     \end{equation*}  Conclusion i) is proved by \eqref{coeff.con.1} and Lemmas \ref{real.pr.lem}-\ref{cpr.lem}. 

ii) Suppose, on the contrary, that \eqref{global.assump} does not hold for the case $\mathcal K(f)\ne\emptyset$. Then $-\infty<\kappa_-+2\le \kappa_+<\infty$ and there exists $k_0\in [\kappa_-+1, \kappa_++L-3]$ such that
\begin{equation}\label{contra.assum}
\sum_{k_0-L+2\le n_1<n_2\le k_0  }|\Im(c_{n_1}\overline{c_{n_2}})|= 0.
\end{equation}

For the case that $k_0\in [\kappa_-+1, \kappa_+-1]$,  there exists $n_0\in [k_0-L+2, k_0]$ such that $c_{n_0}\ne 0$ by \eqref{f.real.cond}. Set $\theta_0=-\arg c_{n_0}\in[-\pi,\pi)$. By  \eqref{coeff.eq.1}-\eqref{kf.def} and \eqref{contra.assum}, we have  $\kappa_-<n_0<\kappa_+$,
\begin{equation*}
\Im (e^{i\theta_0}c_k)=0{\rm\ for\ all\ }k\in [k_0-L+2, k_0],\end{equation*} and
\begin{equation*}
\Im( e^{i\theta_0}c_{k_1})\ne 0{\rm \ and\ }\Im (e^{i\theta_0}c_{k_2})\ne 0  {\rm \ for\ some\ } k_1\le\kappa_-< k_0<\kappa_+\le k_2.
  \end{equation*}
  This implies that $\Im (e^{i\theta_0}f)$ is nonzero and  not phase retrievable in $\mathcal S_\R(\phi)$ by Lemma \ref{real.pr.lem}. It contradicts the conjugate phase retrievability of $f$ by Lemma \ref{cpr.lem}, and hence  \eqref{global.assump} holds.

  For the case that $\kappa_+\le k_0\le \kappa_++L-3$, we have \begin{equation*}
  c_l=0{\rm\ for\ all\ } \kappa_+\le l\le k_0,
   \end{equation*}
   and if $k_0-L+2\le\kappa_-$, $c_l=0$ for all $k_0-L+2\le l\le \kappa_-$ and $\Im c_l \overline{c_{\kappa_+-1}}=0$ for all $\kappa_-+1\le l\le \kappa_+-1$, otherwise $\Im c_l \overline{c_{\kappa_+-1}}=0$ for all $k_0-L+2\le l\le \kappa_+-1$. This together with  \eqref{coeff.eq.1}-\eqref{kf.def} and \eqref{coeff.eq.0} implies that there exist $l_1\ge k_0+1$ and $l_2\le k_0-L+1$ such that $\Im c_{l_1} \overline{c_{\kappa_+-1}}\ne 0 $ and  $\Im c_{l_2} \overline{c_{\kappa_+-1}}\ne 0 $. Set $\theta_1=-\arg c_{\kappa_+-1}$. We have $\Re (e^{i\theta_1}f)$ is not phase retrievable, which contradicts the conjugate phase retrievability of $f$ by Lemma \ref{cpr.lem}, and hence  Conclusion ii) holds.

$\Longleftarrow:$  For the case that $\kappa_-=K_+$, $f$ is real-valued up to a unimodular constant. The conjugate phase retrievability of $f$ follows by Lemma \ref{real.pr.lem}. Now we assume that $\kappa_-<K_+$.
Define the isomorphism
\begin{equation}\label{map.def}
\mathcal A: \C\ni z\longmapsto \begin{pmatrix} \Re z \\ \Im z
 \end{pmatrix}\in  \R^2
 \end{equation}
 between  $\C$ and $\R^2$. Then the function $f(x)=\sum_{k\in\Z}c_k\phi(x-k)$ is conjugate phase retrievable in $\mathcal S(\phi)$ if and only if the $\R^2-$valued  function \begin{equation*}(\mathcal A f)(x):=\mathcal A (f(x))=\begin{pmatrix}\Re f(x) \\ \Im f(x)\end{pmatrix}, x\in\R,\end{equation*} is determined up to an orthogonal transform on $\R^2$ if and only if the sequence
 $(\mathcal Ac_{k})_{k\in\Z}$ is  determined up to an orthogonal matrix in $\R^{2\times2}$.
 By the local conjugate phase retrieval of the function $f$ on the intervals $I_j, j\in\Z$, and
  the restrictions
\begin{equation*}
f|_{I_j}(x)=\sum_{k=j-L+1}^{j}c_k\phi(x-k), x\in I_j,
\end{equation*} we have  that the $\R^2$-valued function $\mathcal A f$ and the coefficient vectors  $\mathcal A {c}_k, j-L+1\le k\le j$, can be locally determined, up to an orthogonal matrix ${\bf U}_j$, from the phaseless measurements $|f(x)|, x\in I_j$.

By \eqref{coeff.eq.1}, we have $c_{\kappa_-+1}\ne 0$. Without loss of generality, we assume that ${\bf U}_{\kappa_-+1}={\bf I}_2$, where ${\bf I}_n$ is the identity matrix of size $n\times n$, otherwise replacing $\mathcal A f$ by ${\bf U}^{-1}_{\kappa_-+1}\mathcal A f$. Then we have
\begin{equation}\label{hypothesis.1}
  c_k, \kappa_--L+2\le k\le \kappa_-+1, {\rm \ can\ be\ determined\ from\ } |f(x)|, x\in  I_{\kappa_-+1},\end{equation}
  and
  \begin{equation*}
  K_+(f|_{\cup_{k\le \kappa_-}I_k})\ge \kappa_--L+2
  \end{equation*} by \eqref{f.real.cond}.
  As the restriction $f|_{\cup_{k\le\kappa_-}I_k}$ on the interval $\cup_{k\le \kappa_-}I_k$ is real-valued up to a unimodular constant, we have
  \begin{equation}\label{hypothesis.2}
  c_k, k\le \kappa_-, {\rm \ can\ be\  determined\ from\ } |f(x)|, x\in \cup_{k\le \kappa_-} I_k,\end{equation} by \eqref{f.real.cond} and Lemma \ref{real.pr.lem}.  Next  we prove the following claim:
\begin{equation*}\label{induct.claim}
\mathcal A c_k, k\ge\kappa_-+1 {\rm \ can\ be\  determined\ from\ } |f(x)|, x\in \cup_{k> \kappa_-} I_k,
\end{equation*}
by induction.  Inductively we assume that $\mathcal A c_k,   k\le K$ are  determined.
The inductive proof is completed if $K\ge K_+$. Otherwise $K< K_+$.
 By  the local conjugate phase retrieval on the interval $ I_{K+1}$,  the coefficient vectors $\mathcal A c_k, K-L+2\le k\le K+1$, are determined up to an orthogonal transform on $\R^2$. Denote the coefficient vectors  recovered from the phaseless measurements $|f(x)|, x\in I_{K+1},$ by $\mathcal A d_k, K-L+2\le k\le K+1$.

\vskip 2mm
{\bf Case 1}: $\kappa_-<K<\kappa_++L-2$

 By  \eqref{global.assump}, there exist  entries $K-L+2\le k_1,k_2\le K$ such that $\Im c_{k_1}\overline{c_{k_2}}\ne0$. It implies that the vectors $\mathcal Ac_{k_1}, \mathcal Ac_{k_2}$ form a basis for $\R^2$.
   Then we have ${\bf U}_{K+1}=\begin{pmatrix}\mathcal A d_{k_1} &\mathcal A d_{k_2}\end{pmatrix}\begin{pmatrix}\mathcal A c_{k_1} &\mathcal A c_{k_2}\end{pmatrix}^{-1}$ is an orthogonal matrix in $\R^{2\times 2}$ and  \begin{equation*}
\mathcal A{ d}_{k}={\bf U}_{K+1}\mathcal A{ c}_{k}, {\rm\ for\ all\ }K-L+2\le k\le K+1.
\end{equation*}
Hence $\mathcal A c_{K+1}={\bf U}^{-1}_{K+1}\mathcal Ad_{K+1}$ is determined.
\vskip 2mm
{\bf Case 2:} $K>\kappa_++L-3$

Observe that the restriction $f|_{\cup_{k\ge \kappa_++L-1}I_{k}}$ is  real-valued  up to a unimodular constant and
\begin{equation*} K_-(f|_{\cup_{k\ge \kappa_++L-1}I_{k}})\le \kappa_++L-2\le K.\end{equation*} By \eqref{f.real.cond} and Lemma \ref{real.pr.lem},
 we have $c_{K+1}$ and $\mathcal Ac_{K+1}$ is uniquely determined.

 This together with the inductive hypothesis and \eqref{hypothesis.1} implies that $\mathcal A c_k,  \kappa_-+1 \le k\le K+1$, can be uniquely determined. The inductive argument can proceed.
Combining \eqref{hypothesis.1} and \eqref{hypothesis.2}, it completes the proof.
\end{proof}

\section{Conjugate phaseless (Hermite) sampling and reconstruction in a complex shift-invariant space}\label{cps}
In this section we consider a conjugate phaseless (Hermite) sampling and reconstruction problem  which determines a function $f\in \mathcal S(\phi)$, up to a unimodular constant and conjugation, from its phaseless Hermite samples $|f(\gamma)|, \gamma\in\Gamma+\Z$, and $|f'(\gamma)|,\gamma\in\Gamma'+\Z$, taken on the shift-invariant sets $\Gamma+\Z$ and $\Gamma'+\Z$ with finite densities respectively. Here  the sampling density of a set $\Lambda\subset \R$  is defined by
\begin{equation*}\label{samplingrate.def}
D(\Lambda):=\lim_{b-a\to \infty}   \frac{\#(\Lambda\cap [a, b])}{b-a},
\end{equation*}
where   $\# E$ is the cardinality of a set $E$.
\subsection{Conjugate phaseless sampling and reconstruction in $\mathcal S(\phi)$}\label{sampling.sec}
 Given a compactly supported function $\phi$ and an open set $A\subset \R$, we say $\phi$ has local linear independence on $A$ if $\sum_{k\in\Z}c_k\phi(x-k)=0$ for  all $x\in A$ implies that $c_k=0$ for all $k\in K_A$ \cite{Sun10}, and define
\begin{equation*}
\Phi_A(x):=\left(\phi(x-k)\right)_{k\in K_A}, x\in A,
\end{equation*} and
\begin{equation*}
W_{\phi,A}:=\span\{\Phi_A(x)(\Phi_A(x))^T, x\in A\}\subset \mathbb H_L,
\end{equation*} where $ K_A=\{k, \phi(x-k)\not\equiv 0{\rm \ on\ }A\}$ and $\mathbb H_L$ be the space of all symmetric real matrices of size $L\times L$.
One may easily verify that  the local linear independence of the generator $\phi$ on $A$ is necessary for the local conjugate phase retrieval of $\mathcal S(\phi)$ on $A$.
In the following theorem, we provide a sufficient condition  for the local conjugate phase retrieval of $\mathcal S(\phi)$ and  the  conjugate phaseless sampling and reconstruction in $\mathcal S(\phi)$.

\begin{theorem}\label{sufficiency.thm}
{\rm Let $\phi$ be a real-valued continuous function satisfying \eqref{supportlength.def2} and \eqref{L.assump}.
Assume that $\phi$ has local linear independence on the intervals $I_j=(j,j+1), j\in\Z$, and \begin{equation}\label{spanning.prop}W_{\phi,(0,1)}=\mathbb H_L.
 \end{equation}Then the shift-invariant space $\mathcal S(\phi)$ is locally conjugate phase retrievable on $I_j, j\in\Z$, and  there exists a finite set $\Gamma\subset (0,1)$ of $\frac{L(L+1)}{2}$ distinct points such that any nonzero function $ f\in \mathcal S(\phi)$ with \eqref{f.real.cond} and \eqref{global.assump} can be determined, up to a unimodular constant and conjugation, from the phaseless samples $|f(\gamma)|, \gamma\in\Gamma+\Z$, taken on the set $\Gamma+\Z$.}
\end{theorem}
\begin{proof}
By Theorem \ref{global.thm}, we only need to find a finite set $\Gamma\subset(0,1)$ such that any function in $ \mathcal S(\phi)$ is  local conjugate phase retrieval on the intervals $I_j=(j,j+1), j\in\Z$, from the samples taken on the set $\Gamma+j$. For any function $f(x)=\sum_{k\in\Z}c_k\phi(x-k)\in \mathcal S(\phi)$ and $j\in\Z$,
$f$
is locally conjugate phase retrievable on $I_j$ if and only if $(\mathcal Ac_{j-L+1},\ldots, \mathcal Ac_{j})$ can be  determined up to  an orthogonal matrix in $\R^{2\times 2}$
 if and only if the associated  Gram matrix $(\langle \mathcal Ac_n, \mathcal Ac_{n'}\rangle)_{j-L+1\le n,n'\le j}$ is uniquely determined,
  where $\mathcal A$ is the isomorphism in \eqref{map.def} and $L\ge 3$, see  Lemma 4.11 in  \cite{CCS22}.

 By \eqref{spanning.prop}, there exists a set $\Gamma=\{\gamma_l, 1\le l\le \frac{L(L+1)}{2}\}\subset(0,1)$  such that the outer products $\Phi(\gamma_l)(\Phi(\gamma_l))^T,  1\le l\le \frac{L(L+1)}{2}$, form a basis for ${\mathbb H}_L$. Let ${\bf E}_{mn}, 1\le m\le n\le L$, be the matrices with   all entries taking value zero  except value one at the $mn$-th and $nm$-th entries, and the set $\{{\bf F}_{n}, 1\le n\le \frac{L(L+1)}{2}\}:=\{{\bf E}_{11}, {\bf E}_{12}, \ldots, {\bf E}_{LL} \}$ forms a standard basis for
  ${\mathbb H}_L$. There exists  a nonsingular matrix ${\bf A}=(a_{ln})\in\R^{\frac{L(L+1)}{2}\times\frac{L(L+1)}{2}}$ such that
 \begin{equation*}
 \Phi(\gamma_l)(\Phi(\gamma_l))^T=\sum_{n=1}^{L(L+1)/2}a_{ln}{\bf F}_n,  1\le l\le \frac{L(L+1)}{2}.
 \end{equation*}
As
 \begin{equation*}
 |f(\gamma_l+j)|^2=\sum_{j-L+1\le n,n'\le j}c_n\overline{c_{n'}}\phi(\gamma_l-n)\phi(\gamma_l-n'), 1\le l\le \frac{L(L+1)}{2},
 \end{equation*} the inner products $\langle\mathcal Ac_n, \mathcal A c_{n'}\rangle=\Re(c_n\overline{c_{n'}}), j-L+1\le n,n'\le j$, in the Gram matrix can be uniquely determined by ${\bf A}^{-1}(|f(\gamma_l+j)|^2)_{1\le l\le \frac{L(L+1)}{2}}$. Then we have that the function $f$ is locally conjugate phase retrievable on $I_j$ and complete the proof.
\end{proof}
    The complex shift-invariant space $\mathcal S(\phi)$ generated by a compactly supported function $\phi$ is locally finite-dimensional. Local conjugate phase retrieval on a bounded open set in $\mathcal S(\phi)$ can be seen as a finite-dimensional conjugate phase retrieval. Observe that a complex point can be determined from its distances to three points which are not colinear \cite{GSWX18}, and $(c_0,c_1,c_2)\in\C^3$ is determined up to a unimodular constant and conjugation only if the inner products $\langle \mathcal Ac_i, \mathcal Ac_j\rangle, {0\le i\le j\le 2}$, are known.  We remark that the sampling density $\frac{L(L+1)}{2}$ of the set $\Gamma+\Z$ and the spanning condition \eqref{spanning.prop} in the above theorem may not be necessary for local conjugate phase retrieval in $\mathcal S(\phi)$ with  large $L$, while they are necessary and sufficient for the case that $L=3$,
see an equivalent formulation for the conjugate phase retrieval of vectors in $\C^L$ \cite{Lai20}.
\begin{corollary}\label{l3.cor}
{\rm Let $\phi$ be a real-valued continuous function with support length $L=3$. If any function $f\in \mathcal S(\phi)$ satisfying \eqref{f.real.cond} and \eqref{global.assump} can be determined, up to a  unimodular constant and conjugation, from the magnitudes $|f(x)|, x\in\R$, then $W_{\phi, [0,1]}=\mathbb H_3$.}
\end{corollary}

\begin{example}\label{phi1.ex}
{\rm Let \begin{equation*}\label{phi_0.def}
\phi_1(x)=\left\{\begin{array}{ll}
x^3/2&{\rm if\ }0\le x<1\\
-x^3+3x^2-2x+1/2&{\rm if\ }1\le x<2\\
x^3/2-3x^2+5x-3/2&{\rm if\ }2\le x<3\\
0&{\rm otherwise}.
\end{array}\right.
\end{equation*}  The function $\phi_1$ has local linear independence on the interval $(0,1)$ and
 $W_{\phi_1,(0,1)}=\mathbb H_3$\cite{CJS19}. We can find a finite set $\Gamma\subset(0,1)$  so  that the corresponding outer products $\Phi_{(0,1)}(\gamma)(\Phi_{(0,1)}(\gamma))^T$, $\gamma\in \Gamma$, form a basis for  $\mathbb H_3$. By Theorem \ref{sufficiency.thm}, any function $f\in V(\phi_1)$ satisfying \eqref{f.real.cond} and \eqref{global.assump}
can be determined, up to a unimodular constant and conjugation,  from the phaseless samples $|f(\gamma)|, \gamma\in\Gamma+\Z$.}
\end{example}
Next we provide a counterexample that  the space $\mathcal S(\phi)$ is not locally conjugate phase retrievable on the intervals $I_j=(j,j+1), j\in\Z$,  when the generator $\phi$ is supported on $[0,3]$ and $W_{\phi, (0,1)}\subsetneq\mathbb H_3$.
\begin{example}\label{b3.ex}
{\rm Let $B_3$  be the B-spline of order $3$ and write
 \begin{equation*}\label{B3.def}
 B_3(x)=
\begin{cases}
\frac{1}{2}x^2 &{\rm\  if\ }0\le x<1\\
-x^2+3x-\frac{3}{2} &{\rm\  if\ }1\le x<2\\
\frac{1}{2}x^2-3x+\frac{9}{2} &{\rm\  if\ }2\le x<3\\
0 &{\rm\ otherwise}.
\end{cases}
 \end{equation*}
 Set ${\bf B}_{(0,1)}(x)=(B_3(x),B_3(x+1),B_3(x+2))^T,\ 0< x<1$. Then we have
 \begin{equation*}
 {\bf B}_{(0,1)}(x)=\frac{1}{2}\left(
\begin{array}{cccc}
 0\\
 1\\
 1
\end{array}
\right )+\left(
\begin{array}{cccc}
 0\\
 1\\
 -1
\end{array}
\right )x+\frac{1}{2}\left(
\begin{array}{cccc}
 1\\
 -2\\
 1
\end{array}
\right )x^2
 \end{equation*}
 and
 \begin{equation}\label{B3}
 \begin{aligned}
 {\bf B}_{(0,1)}(x){\bf B}_{(0,1)}(x)^T&=\frac{1}{4}\left(
\begin{array}{cccc}
 0& 0& 0\\
 0& 1& 1\\
 0& 1& 1
\end{array}
\right )+\left(
\begin{array}{cccc}
 0& 0& 0\\
 0& 1& 0\\
 0& 0& -1
\end{array}
\right )x+\frac{1}{4}\left(
\begin{array}{cccc}
 0& 1& 1\\
 1& 0& -5\\
 1& -5& 6
\end{array}
\right )x^2\\
&+\frac{1}{2}\left(
\begin{array}{cccc}
 0& 1& -1\\
 1& -4& 3\\
 -1& 3& -2
\end{array}
\right )x^3+\frac{1}{4}\left(
\begin{array}{cccc}
 1& -2& 1\\
 -2& 4& -2\\
 1& -2& 1
\end{array}
\right )x^4.
\end{aligned}
 \end{equation}
The space $W_{B_3,(0,1)}$ spanned by the outer products ${\bf B}_{(0,1)}(x){\bf B}_{(0,1)}(x)^T, 0<x<1,$ is a 5-dimensional linear subspace of $\mathbb H_3$. While a function $f\in \mathcal B_3$ restricted on an interval $I_j=(j,j+1), j\in\Z,$ is  a polynomial of degree $2$, it may not be locally conjugate phase retrievable on $I_j$ if all the roots of $f|I_j$ are not real.}
\end{example}
\subsection{Conjugate phaseless Hermite sampling and reconstruction in $\mathcal S(\phi)$}\label{hermite.sec}
 Hermite sampling is to interpolate by the samples of the function and its first derivative. For the B-spline $B_3$ of order $3$, we have the derivative
\begin{equation*}
  B'_3(x)=
\begin{cases}
x &{\rm\  if\ }0\le x<1\\
-2x+3 &{\rm\  if\ }1\le x<2\\
x-3 &{\rm\  if\ }2\le x<3\\
0 &{\rm\ elsewhere},
\end{cases}
 \end{equation*} and for all $x\in (0,1)$,
 \begin{equation*}
 {\bf B'}_{(0,1)}(x):=(B'_3(x),B'_3(x+1),B'_3(x+2))^T=\left(
\begin{array}{cccc}
 0\\
 1\\
 -1
\end{array}
\right )+\left(
\begin{array}{cccc}
 1\\
 -2\\
 1
\end{array}
\right )x
 \end{equation*}
 and
 \begin{equation*}\label{B3'}
 {\bf B}'_{(0,1)}(x){\bf B}'_{(0,1)}(x)^T=\left(
\begin{array}{cccc}
 0& 0& 0\\
 0& 1& -1\\
 0& -1& 1
\end{array}
\right )+\left(
\begin{array}{cccc}
 0& 1& -1\\
 1& -4& 3\\
 -1& 3& -2
\end{array}
\right )x+\left(
\begin{array}{cccc}
 1& -2& 1\\
 -2& 4& -2\\
 1& -2& 1
\end{array}
\right )x^2.
 \end{equation*}
 This together with  \eqref{B3} implies that \begin{equation*}W_{B_3,(0,1)}\cup W_{B'_3,(0,1)}=\mathbb H_3,\end{equation*} and  there exist two  subsets $\Gamma, \Gamma'\subset (0,1)$ such that the set \begin{equation*}\{{\bf B}_{(0,1)}(\gamma){\bf B}_{(0,1)}(\gamma)^T,\gamma\in\Gamma\}\cup\{{\bf B}'_{(0,1)}(\gamma'){\bf B}'_{(0,1)}(\gamma')^T,\gamma'\in\Gamma'\} \end{equation*}forms a basis of $\mathbb H_3$.
 Following a similar argument as in Theorem \ref{sufficiency.thm}, we have $\mathcal B_3$ can be local conjugate phase retrieval on $I_j=(j,j+1), j\in\Z,$ from  Hermite samples, and hence any function $f\in \mathcal B_3$ satisfying \eqref{f.real.cond} and \eqref{global.assump} can be determined, up to a unimodular constant and conjugation, from the phaseless Hermite samples $|f(\gamma)|, \gamma\in\Gamma+\Z,$ and $|f'(\gamma')|,\gamma'\in\Gamma'+\Z$.

For a general real-valued differentiable generator $\phi$ with compact support, by  a similar argument as in  Theorem \ref{sufficiency.thm}, we have the following result on the conjugate phaseless Hermite sampling and reconstruction  in $\mathcal S(\phi)$.

\begin{theorem}\label{hermite.thm}
{\rm Let $\phi$ be a real-valued differentiable function satisfying \eqref{supportlength.def2} and \eqref{L.assump}.
Assume that the functions $\phi$ and $\phi'$ have local linear independence on the interval $(0,1)$ and $W_{\phi,(0,1)}\cup W_{\phi',(0,1)}={\mathbb H}_L$. Then there exist finite sets $\Gamma,\Gamma'\subset (0,1)$ such that any function $f\in \mathcal S(\phi)$ satisfying \eqref{f.real.cond} and \eqref{global.assump} can be determined, up to a unimodular constant and conjugation, from the phaseless Hermite samples $|f(\gamma)|$, $\gamma\in\Gamma+\Z$, and $|f'(\gamma')|$, $\gamma'\in\Gamma'+\Z$.}
\end{theorem}
\begin{proof}The proof is similar to that of Theorem \ref{sufficiency.thm}, we omit it here.\end{proof}
\section{Conjugate phaseless Hermite sampling and reconstruction  in a complex spline space}\label{cps.b}
Let $\mathcal B_N$ be the spline space as in \eqref{spline.sis.def} with $N\ge 3$.
For any function $f\in \mathcal B_N$ and $j\in\Z$, the restriction $f|_{I_j}$ on the interval $I_j=(j,j+1)$ is a polynomial of order $N-1$, which may not be conjugate phase retrieval from the phaseless observations $|f(x)|, x\in I_j$. Similar to the conjugate phaseless Hermite sampling and reconstruction in $\mathcal B_3$, we consider  the conjugate phase retrieval in $\mathcal B_N$ from the Hermite measurements for general $N\ge 3$ in this section.

\subsection{Conjugate phaseless Hermite sampling and reconstruction in $\mathcal B_N$}
\begin{theorem}\label{hermite.spline.thm}
{\rm Let $\mathcal B_N$ be the spline space as in \eqref{spline.sis.def} with $N\ge 3$ and
 the sets $\Gamma,\Gamma'$ contain $2N-1$ and $2N-5$ distinct points in $(0,1)$ respectively. Any function $f\in \mathcal B_N$ satisfying \eqref{f.real.cond} and \eqref{global.assump} can be determined, up to a unimodular constant and conjugation, from the phaseless Hermite samples $|f(\gamma)|, \gamma\in\Gamma+\Z$, and $|f'(\gamma')|,\gamma'\in\Gamma'+\Z$.}
\end{theorem}
We remark that  the sampling density for the conjugate phase retrieval in the above theorem is $4N-6\le\frac{N(N+1)}{2}$ for $N\ge 3$, cf. a generic conjugate phase retrievable frame of $4N-6$ real vectors for $\C^N$ \cite{Lai20}.

To prove Theorem \ref{hermite.spline.thm}, we need the following lemma on the conjugate phase retrieval in $\C^N$, see an equivalent statement of conjugate phase retrieval for polynomials in \cite{McDonald04}. Different with  the complex methods for the entire functions used in \cite{McDonald04}, at the end of this subsection we provide an alternative proof based on the characterization of  conjugate phase retrievable frames in \cite{CCS22}.

 \begin{lemma}\label{discrete.hermite.lem}
 {\rm Let  $\Gamma=\{\gamma_n, 1\le n\le2N-1\}\subset \R $ and $\Gamma'=\{\gamma'_n, 1\le n\le 2N-5\}\subset\R$, and the $N$-dimensional real vectors $\psi_{n}=(1, \gamma_n, \gamma_n^2,\ldots, \gamma_n^{N-1})^T, 1\le n\le 2N-1,$ and $\tilde\psi_n=(0, 1, 2\gamma'_n, \ldots, (N-1)(\gamma'_n)^{N-2})^T, 1\le n\le 2N-5$. Then the set \begin{equation}\label{real.frame.eq}\Psi=\{\psi_n, 1\le n\le 2N-1\}\cup\{\tilde\psi_{n}, 1\le n\le2N-5\}\end{equation} does conjugate phase retrieval in $\C^N$.}
\end{lemma}

\begin{proof}[Proof of Theorem \ref{hermite.spline.thm}]
By Theorem \ref{global.thm}, it reduces to proving that $\mathcal B_N$  is locally conjugate phase retrievable on the intervals $I_j=(j,j+1), j\in\Z$.

 Observe that for any function $f\in \mathcal B_N$ and $j\in\Z$, the restriction $f|_{I_j}$ on the interval $I_j$ is a polynomial of order $N-1$. By Lemma \ref{discrete.hermite.lem}, the function $f|_{I_j}$ can be determined, up to a unimodular constant and conjugation, from the phaseless Hermite samples $|f(\gamma)|,\gamma\in\Gamma+j$, and $ |f'(\gamma')|, \gamma'\in\Gamma'+j$. This proves our conclusion.
\end{proof}

To prove Lemma \ref{discrete.hermite.lem}, we recall a characterization of conjugate phase retrieval in a complex range space, and provide a proposition on the unique determination of a complex number.
\begin{lemma}{\rm\label{discrete.lem}\cite{CCS22}
Let ${\bf A}^T=({\bf a}_1,\ldots,{\bf a}_M)\subset\R^{M\times N}$ be a matrix of rank $N$, ${\mathcal N}_{\bf A}$ be the set of  real matrices ${\bf X}$ of size $N\times N$ such that ${\rm Tr}( {\bf a}_i {\bf a}_i^T  {\bf X} ) =0$ for  all $1\le i\le M$, and ${\bf M}_{N,2}(\R)$ be the set of all  real matrices of size $N\times N$ with rank at most 2, which can be written as \begin{equation*}
{\bf M}_{N,2}(\R)=\{\Re({\bf xy}^*), {\bf x}, {\bf y}\in\C^{N}\}=\{\Re({\bf xx}^*-{\bf yy}^*)+\Re({\bf yx}^*-{\bf xy}^*), {\bf x}, {\bf y}\in\C^{N}\}.
\end{equation*}
 Then
the complex range space
$R_{\C}({\bf A}):=\{{\bf Ax}, {\bf x}\in\C^N\}$ is conjugate phase retrievable if and only if
all matrices ${\bf X}$ in ${\mathcal N}_{\bf A}\cap {\bf M}_{N,2}(\R)$ are skew-symmetric.
}
\end{lemma}

As $\Re(z\overline{z_0})=\Re z\Re z_0+\Im z\Im z_0$ for any $z, z_0\in\C$,  we have
\begin{equation*}\label{unique.eq}
\Im z=\frac{\Re(z\overline{z_0})-\Re z\Re z_0}{\Im z_0}
\end{equation*}
 provided that  $\Im z_0\ne 0$. Hence we have the following result on the determination of a complex number.
\begin{proposition}\label{unique.prop}
{\rm  Let $z_0$ be a complex number with $\Im z_0\ne 0$. Then $z$ can be uniquely determined from  $\Re z$ and $\Re (z\overline{z_0})$.}
\end{proposition}

\begin{proof}[Proof of Lemma \ref{discrete.hermite.lem}] Observe that for any $ {\bf x}, {\bf y}\in\C^{N}$, $\Re({\bf xx}^*-{\bf yy}^*)$ is symmetric and $\Re({\bf yx}^*-{\bf xy}^*)$ is skew-symmetric.  By Lemma \ref{discrete.lem}, our proof reduces to showing that
\begin{equation}\label{frame.measure.1}
\Tr(\psi_n\psi_n^T\Re({\bf xx}^*-{\bf yy}^*))=0 {\rm \ for \ all\ } 1\le n\le 2N-1
\end{equation}
and
\begin{equation}\label{frame.measure.2}
\Tr(\tilde\psi_n\tilde\psi_n^T\Re({\bf xx}^*-{\bf yy}^*))=0 {\rm \ for \ all\ } 1\le n\le 2N-5,
\end{equation}
if and only if  \begin{equation}\label{frame.eq.3}\Re({\bf xx}^*-{\bf yy}^*)={\bf O},\end{equation}
where $\bf O$ is the zero matrix.

The sufficiency is trivial. Now we  prove the necessity.
 For a vector ${\bf s}=(s_0,\ldots, s_{N-1})\in\C^N$, denote the $j$-th entry of $\bf s$ by ${\bf s}_j, 0\le j\le N-1$, ${\bf s}':=(0,s_1, \ldots, (N-1)s_{N-1})\in \C^N$, $\tilde{\bf s}:=(s_0, s_1,\ldots, s_{N-1}, 0,0, \ldots, 0)\in \C^{2N-1}$
 and \begin{equation*}\label{pf.A.def}
 A_m({\bf s}):=\sum_{j=0}^{m} \tilde{\bf s}_j\overline{\tilde{\bf s}_{m-j}},\ 0\le m\le 2N-2,
 \end{equation*}
 and we have
 \begin{equation}\label{pf.eq.1}
 (N-1)(N-2)A_{2N-3}({\bf s})=A_{2N-3}({\bf s}') {\rm \ and\ } (N-1)^2A_{2N-2}({\bf s})=A_{2N-2}({\bf s}').
 \end{equation}
 Let ${\bf x}=(x_0, x_1,\ldots, x_{N-1}), {\bf y}=(y_0, \ldots, y_{N-1})\in \C^N$ be the vectors satisfying \eqref{frame.measure.1} and \eqref{frame.measure.2}.
 Then  we have
 \begin{equation*}
\psi_n^T\Re({\bf xx}^*)\psi_n=\psi_n^T\Re({\bf yy}^*)\psi_n, 1\le n\le 2N-1,
\end{equation*} and
 \begin{equation*}
\tilde\psi_n^T\Re({\bf xx}^*)\tilde\psi_n=\tilde\psi_n^T\Re({\bf yy}^*)\tilde\psi_n, 1\le n\le 2N-5.
\end{equation*}
Expanding the above equalities, we obtain
\begin{equation*}\label{f.meas.1}
\sum_{m=0}^{2N-2}A_m({\bf x} )\gamma_n^m=\sum_{m=0}^{2N-2}A_m({\bf y})\gamma_n^m, 1\le n\le 2N-1,
\end{equation*}
and
\begin{equation*}\label{f.meas.2}
\sum_{m=2}^{2N-2}A_m({\bf x}' )(\gamma'_n)^m=\sum_{m=2}^{2N-2}A_m({\bf y}')(\gamma'_n)^m, 1\le n\le 2N-5.
\end{equation*}
This together with  the polynomial interpolation and \eqref{pf.eq.1} implies that
\begin{equation}\label{corre.1}
A_m({\bf x})=A_{m}({\bf y}) {\rm \ and\ } A_m({\bf x}')=A_{m}({\bf y}')
\end{equation}
for all  $0\le m\le 2N-2$.

If ${\bf x}={\bf 0}$, then we can easily get that ${\bf y}={\bf 0}$ from \eqref{corre.1},  and hence \eqref{frame.eq.3} holds. Now  we assume that ${\bf x}$ and $\bf y$ are nonzero,  $x_{k_0}$ and $y_{l_0}$ are the first nonzero components of ${\bf x}$ and $\bf y$ respectively.
By $A_{2k}({\bf x})=A_{2k}({\bf y})$ for all $0\le k\le k_0$,  we have
 \begin{equation*}
 k_0=l_0{\rm \ and\  } |x_{k_0}|=|y_{k_0}|.
 \end{equation*}  If $k_0=N-1$, \eqref{frame.eq.3} follows immediately. Now we consider the case that $k_0<N-1$. Without loss of generality, we assume that \begin{equation}\label{pf.assump.eq}
 x_{k_0}=y_{k_0}>0,
 \end{equation} as  $\Re (e^{i\theta}{\bf x}e^{-i\theta}{\bf x}^*)=\Re ({\bf x}{\bf x}^*)$. Then for ${\bf s}\in\{{\bf x}, {\bf y}, {\bf x}', {\bf y}'\}$, we have
 \begin{equation}\label{pf.a.eq.1}
 A_{2k_0+1}({\bf s})=2s_{k_0}\Re s_{k_0+1},
 \end{equation} and  for $2\le k\le 2N-2k_0-2$,
  \begin{equation}\label{pf.a.eq.2}
  A_{2k_0+k}({\bf s})=\left\{\begin{array}{ll}
  2s_{k_0}\Re s_{k_0+k}+2\Re\sum_{j=1}^{(k-1)/2}s_{k_0+j}\overline{s_{k_0+k-j}}& k{\rm \ is\ odd}\\
   2s_{k_0}\Re s_{k_0+k}+2\Re\sum_{j=1}^{k/2-1}s_{k_0+j}\overline{s_{k_0+k-j}}+|s_{k/2}|^2& k{\rm \ is\ even}.
   \end{array} \right.
 \end{equation}
Next we will show that by \eqref{corre.1}, we have
 \begin{equation}\label{pf.eq.3}
 {\bf x}={\bf y} {\rm \ or \ }  {\bf x}=\overline{\bf y}.
 \end{equation}
\vskip 2mm
 {\bf Case 1:} ${\bf x}\in \R^N$.

  In this case, \eqref{pf.eq.3} holds if and only if \begin{equation}\label{pf.eq.4}y_{k_0+l}=x_{k_0+l}\in \R {\rm\ for\ all\ }1\le l\le N-1-k_0,\end{equation}
 if and only if  \begin{equation}\label{pf.eq.5}x_{k_0+l}=\Re y_{k_0+l}{\rm \ \  and\ \  } |x_{k_0+l}|=|y_{k_0+l}|, {\rm \ for\ all\ } 1\le l\le N-1-k_0.\end{equation}  Now we prove \eqref{pf.eq.4} by induction.

 For  $l=1$, by \eqref{pf.assump.eq}-\eqref{pf.a.eq.2}, we have
$x_{k_0+1}=\Re y_{k_0+1}$ from $A_{2k_0+1}({\bf x})=A_{2k_0+1}({\bf y})$,
 and  the linear  system
 \begin{equation}\label{pf.eq.system}\left\{\begin{array}{c}2x_{k_0}(x_{k_0+2}-\Re y_{k_0+2})+(|x_{k_0+1}|^2-|y_{k_0+1}|^2)=0\\
  2k_0(k_0+2)x_{k_0}(x_{k_0+2}-\Re y_{k_0+2})+(k_0+1)^2(|x_{k_0+1}|^2-|y_{k_0+1}|^2)=0\end{array}\right.\end{equation}  from  $A_{2k_0+2}({\bf x})=A_{2k_0+2}({\bf y})$ and $A_{2k_0+2}({\bf x}')=A_{2k_0+2}({\bf y}')$. It implies that   $x_{k_0+2}=\Re y_{k_0+2}$ and $|x_{k_0+1}|=|y_{k_0+1}|$. Then \eqref{pf.eq.5} and hence \eqref{pf.eq.4} holds for $l=1$.

Inductively, we assume that \begin{equation}\label{pf.hypothesis.1}
x_{k_0+k}=y_{k_0+k}\in \R {\rm\ holds\ for\ all\ }k<n. \end{equation} If $n\ge N-k_0$, the inductive proof completes. Otherwise $n< N-k_0$.
By \eqref{pf.a.eq.2}, \eqref{pf.hypothesis.1}
and $A_{2k_0+k}({\bf y})= A_{2k_0+k}({\bf x}),k\le n$, we get \begin{equation}\label{pf.induct.eq.1}\Re y_{k_0+k}=x_{k_0+k}, k\le n.\end{equation}
Then for $n<k\le 2n-1$, we have  $\Re y_{k_0+k}=x_{k_0+k}$  by \eqref{pf.a.eq.2}, \eqref{pf.hypothesis.1}, \eqref{pf.induct.eq.1} and $A_{2k_0+k}({\bf x})=A_{2k_0+k}({\bf y})$.
Therefore, by a similar argument  as in \eqref{pf.eq.system}, we obtain
\begin{equation*}
 |x_{k_0+n}|=|y_{k_0+n}| {\rm\ and\ }x_{k_0+2n}=\Re y_{k_0+2n},
 \end{equation*} from  $A_{2k_0+2n}({\bf x})=A_{2k_0+2n}({\bf y})$ and $A_{2k_0+2n}({\bf x}')=A_{2k_0+2n}({\bf y}')$.
 Together with \eqref{pf.induct.eq.1}, we have \eqref{pf.eq.5}, and hence \eqref{pf.eq.4} holds for $l=n$. Our inductive proof can proceed and hence \eqref{pf.eq.3} holds.

\vskip 2mm
 {\bf Case 2:}  $\Im x_k\ne 0$ for some $k_0<k\le N-1$.

   Let $k_1$ be the first entry such that $\Im c_{k_1}\ne 0$. Following a similar argument as in Case 1, we have
\begin{equation}\label{corre.2}
x_k=y_k \in \R {\rm\  for\ all\ } k_0\le k<k_1,  |x_{k_1}|=|y_{k_1}| {\rm\ and\ }
\Re x_{k}=\Re y_{k}{\rm\  for\ all\ }  k_1\le k\le 2k_1-k_0,
\end{equation}
which implies that $x_{k_1}=y_{k_1}$ or  $x_{k_1}=\overline{y_{k_1}}$. Then  for ${\bf s}\in\{{\bf x}, {\bf y}, {\bf x}', {\bf y}'\}$, we have

  \begin{equation}\label{pf.a.eq.3}
  A_{2k_1+k}({\bf s})=\left\{\begin{array}{ll}
  \tilde{A}({\bf s},k)+2\Re\sum_{j=1}^{(k-1)/2}s_{k_1+j}\overline{s_{k_1+k-j}}& k{\rm \ is\ odd}\\
   \tilde{A}({\bf s},k)+2\Re\sum_{j=1}^{k/2-1}s_{k_1+j}\overline{s_{k_1+k-j}}+|s_{k_1+k/2}|^2& k{\rm \ is\ even},
   \end{array} \right.
 \end{equation} where \begin{equation*}\tilde{A}({\bf s},k)=2s_{k_0}\Re s_{2k_1-k_0+k}+\ldots+2s_{k_1-1}\Re s_{k_1+k+1}+2\Re s_{k_1}\overline{s_{k_1+k}}\end{equation*}
 for all $2\le k\le 2N-2k_1-2$.

For the case that $x_{k_1}=y_{k_1}$, \eqref{pf.eq.3} holds if and only if \begin{equation}\label{pf.eq.6}y_{k_1+l}=x_{k_1+l} {\rm\ for\ all\ }1\le l\le N-1-k_1.\end{equation}
 For  $l=1$, by \eqref{corre.2}, \eqref{pf.a.eq.3},  $A_{2k_1+1}({\bf x})=A_{2k_1+1}({\bf y})$ and $A_{2k_1+1}({\bf x}')=A_{2k_1+1}({\bf y}')$, we have
 \begin{equation}\label{pf.eq.system.3}
 \left\{\begin{array}{c}x_{k_0}\Re( x_{2k_1-k_0+1}-y_{2k_1-k_0+1})=\Re( y_{k_1}\overline{x_{k_1+1}}-x_{k_1}\overline{x_{k_1+1}})\\
  k_0(2k_1-k_0+1)x_{k_0}\Re(x_{2k_1-k_0+1}- y_{2k_1-k_0+1})=k_1(k_1+1) \Re( y_{k_1}\overline{x_{k_1+1}}-x_{k_1}\overline{x_{k_1+1}}),\end{array}\right.\end{equation}
which  implies that \begin{equation*}
\Re (x_{k_1}\overline{x_{k_1+1}})=\Re  (y_{k_1}\overline{y_{k_1+1}}){\rm\ and\ }\Re x_{2k_1-k_0+1}=\Re y_{2k_1-k_0+1}.
 \end{equation*}This together with \eqref{corre.2} and Proposition \ref{unique.prop} implies that  \eqref{pf.eq.6} holds.
Inductively, we assume that \begin{equation}\label{pf.hypothesis.2}
x_{k_1+l}=y_{k_1+l} {\rm\ holds\ for\ all\ }l<n. \end{equation}  If $n\ge N-k_1$, the inductive proof completes. Otherwise $n< N-k_1$. Then  from $l=2$ to $n$, by a similar argument as in \eqref{pf.eq.system.3}, we have \begin{equation}\label{pf.eq.7}
\Re x_{2k_1-k_0+l}=\Re y_{2k_1-k_0+l}{\rm \  and\ }\Re( x_{k_1}\overline{x_{k_1+n}})=\Re (y_{k_1}\overline{y_{k_1+n}}),\end{equation}
 from \eqref{corre.2}, \eqref{pf.a.eq.3}, \eqref{pf.hypothesis.2}, $A_{2k_1+l}({\bf y})= A_{2k_1+l}({\bf x}){\rm\ and\ }A_{2k_1+l}({\bf y}')= A_{2k_1+l}({\bf x}')$ sequentially.
 This proves that  \eqref{pf.eq.6} holds for $l=n$ by Proposition \ref{unique.prop}. Our inductive proof can proceed.

Using a similar argument,   we have
\begin{equation*}y_{k_1+l}=\overline{x_{k_1+l}}, {\rm\ for\ all\ }1\le l\le N-1-k_1,\end{equation*}
 if  $x_{k_1}=\overline{y_{k_1}}$.  This together with \eqref{pf.eq.6} proves \eqref{pf.eq.3}.
\end{proof}

\subsection{Reconstruction algorithms and numerical experiments}
Let  $\psi$ be a real-valued compactly supported differentiable piecewise polynomial of order $N-1$ and have  support length $N$ satisfying \eqref{supportlength.def2}, \eqref{L.assump}, \eqref{f.real.cond} and \eqref{global.assump}. Denote $\mathcal P_{N-1}$ be the space of polynomials of order not larger than $N-1$. Assume that $\psi$ has local linear independence on $(0,1)$ and there exists an one-to-one correspondence  from $V(\psi)|_{I_j}$ to $\mathcal P_{N-1}$ for any $I_j=(j,j+1), j\in\Z$. Based on the constructive proof of Lemma \ref{discrete.hermite.lem}, we propose an algorithm for the  conjugate phase retrieval in $V(\psi)$ in this section.
 We remark that the algorithm can be applied for the conjugate phase retrieval in $\C^N$ from the real frame vectors as in \eqref{real.frame.eq}.
~\\
\begin{breakablealgorithm}\label{alg1}
	\caption{ Conjugate phase retrieval
of piecewise polynomials in  $V(\psi)$}
	\label{alg.cpr}
	\begin{algorithmic}
\REQUIRE order $N$, sampling sets $\Gamma=\{\gamma_k\}_{k=0}^{ 2N-2}\subset (0,1)$ and $\Gamma'=\{\gamma'_k\}_{k=0}^{2N-6}\subset (0,1)$,  phaseless measurements $|f(\gamma+j)|,\gamma\in \Gamma$, and $|f'(\gamma'+j)|, \gamma'\in\Gamma'$, $j\in\Z$.

\hskip -0.1in{\bf Initials:} {Compute the Vandermonde matrices ${\bf V}_{\Gamma}=(\gamma_k^m)_{0\le k,m\le 2N-2}$  and ${\bf V}_{\Gamma'}=({\gamma'_k}^m)_{0\le k,m\le 2N-6}$.  For any $j\in\Z$, set the vectors ${\bf d}_j=(d_{j,0}, d_{j,1},...,d_{j, N-1})\in\C^N$, ${\bf d}_j'=(0, d_{j,1},...,(N-1)d_{j, N-1})\in\C^N$ and $\pmb\psi_j(x)=(\psi(x-j+N-1),\ldots, \psi(x-j))^T$, and compute the invertible matrix ${\bf H}_j$ such that $\pmb\psi_j={\bf H}_j(1,x\ldots, x^{N-1})^T$.}

\hskip -0.1in{\bf Steps:}
\STATE 1). {\bf Local conjugate phase retrieval on $I_j=(j,j+1)$.}

\STATE(1a).
{Determine $A_m({\bf d}_j), 0\le m\le 2N-2$, $A_m({\bf d}'_j), 2\le m\le 2N-2$ by the inverse of Vandermonde matrices ${\bf V}_\Gamma$, ${\bf V}_{\Gamma'}$ and the phaseless samples $|f(\gamma+j)|,\gamma\in \Gamma$, and  $|f'(\gamma'+j)|, \gamma'\in\Gamma'$.
}
\STATE(1b). {Find the first entry $k_0$ such that $A_{2k_0}({\bf d}_j)\neq0$
and set $$d_{j,k_0}=\sqrt{A_{2k_0}({\bf d}_j)}\in \R^+.$$
For $1\le l\le k_0+N-1$, determine $\Re d_{j,k_0+l}$ and $|d_{j,k_0+l}|$ from $A_{2k_0+l'}({\bf d}_j)$ and $A_{2k_0+l'}({\bf d}_j')$, $1\le l' \le 2l$.
}
\STATE(1c). {Find the first entry $k_1$ such that $|\Re d_{j,k_1}|\neq|d_{j,k_1}|$, and set  $$\ d_{j,k_1}=\Re d_{j,k_1}+ i\sqrt{|d_{j,k_1}|^2-\Re d_{j,k_1}^2}.$$
For $1\le l\le N-k_1-1$, determine $\Re d_{j,k_1}\overline{d_{j,k_1+l}}$ and $\Re d_{j, k_1+l}$ from $A_{2k_0+l'}({\bf d}_j')$ and $A_{2k_0+l'}({\bf d}_j)$, $l'\le 2k_1-2k_0+l$, and set
$$d_{j,k_1+l}=\Re d_{j,k_1+l}+i\frac{\Re d_{j,k_1}\overline{d_{j,k_1+l}}-\Re d_{j,k_1}\Re d_{j,k_1+l}}{\Im d_{j,k_1}}.$$
}

\STATE (1d). {Set $\tilde{\bf c}_j={\bf H}_j^{-1}{\bf d}_j$ and write $f_{\epsilon}|_{I_j}(x)=\tilde{\bf c}_j^T\pmb\psi_j(x)$.}
\STATE 2).  {\bf Sewing the coefficients among neighbouring intervals}
Let $\tilde{\bf c}_{j_0}$ be a nonzero vector with $\Im (\tilde c_{j_0,k_1}\overline {\tilde c_{j_0,k_2}})\ne 0$ for some $1\le k_1\le k_2\le N-2$. Set $c_{\epsilon}(j_0-N+k+1)=\tilde c_{j_0,k}$ for all $0\le k\le N-1$.
\STATE (2a).
{For any $j\ge j_0$, find the entries  $1\leq k_1, k_2\leq N-1$ satisfying $\Im \tilde c_{j,k_1}\overline{\tilde c_{j,k_2}}\ne0$.
Adjust phase and conjugation of $\tilde{\bf c}_{j+1}$ appropriately to be  ${\bf c}_{j+1}$ so that the vector  $(c_{j+1,k_1-1},c_{j+1,k_2-1})$ satisfy
\begin{equation*}
(c_{j+1,k_1-1},c_{j+1,k_2-1}) =(\tilde c_{j,k_1},\tilde c_{j,k_2})
\end{equation*}
and set $c_\epsilon(j+1)=c_{j+1,N-1}.$
}
\STATE (2b).
{For any $j\le j_0$, find the entries  $0\leq k_1, k_2\leq N-2$ satisfying $\Im\tilde c_{j,k_1}\overline{\tilde c_{j,k_2}}\ne0$.
Adjust phase and conjugation of $\tilde{\bf c}_{j-1}$ appropriately to be ${\bf c}_{j-1}$ so that the vector  $(c_{j-1,k_1+1},c_{j-1,k_2+1})$ satisfy
\begin{equation*}
(c_{j-1,k_1+1},c_{j-1,k_2+1}) =(\tilde c_{j,k_1},\tilde c_{j,k_2})
\end{equation*}
and set $c_\epsilon(j-N)=c_{j-1,0}.$
}
\ENSURE{$f_\epsilon(x)=\sum_{k\in\Z}c_{\epsilon}(k)\psi(x-k)$.
}
	\end{algorithmic}
\end{breakablealgorithm}
~\\

We remark that the above algorithm is not numerical stable as the inverse of a Vandermonde matrix is not numerical stable\cite{Gautschi20}.
Next, we demonstrate the performance of the above algorithm on reconstructing a differentiable function
\begin{equation}\label{f}
f(t)=\sum_{k=K_1}^{K_2}c_k\psi(t-2k)
\end{equation}
with finite duration, where
\begin{equation*}
\psi(t)=
\begin{cases}
\frac{1}{2}t^2 &{\rm\  if\ }0\le t<2\\
-t^2+3t &{\rm\  if\ }2\le t<4\\
\frac{1}{2}t^2-3t &{\rm\  if\ }4\le t<6\\
0 &{\rm\ otherwise}.
\end{cases}
\end{equation*}

Our noisy phaseless samples are taken on $(0,2)+2\Z$,
\begin{equation}\label{samples1}
 z_\epsilon(\gamma)=|f(\gamma)|^2+\|f\|_\infty^2\epsilon(\gamma)\geq0, \gamma\in \Gamma+2\Z\subset (0,2)+2\Z
\end{equation}
and
\begin{equation}\label{samples2}
 z'_\epsilon(\gamma)=|f'(\gamma)|^2+\|f\|_\infty^2\epsilon(\gamma)\geq0, \gamma\in \Gamma'+2\Z\subset (0,2)+2\Z,
\end{equation}
where the sets $\Gamma,\Gamma'$ contain $7$ and $3$ distinct points in $(0,2)$ respectively and  $\epsilon(\gamma)\in[-\varepsilon,\varepsilon]$ are randomly selected with noise level $\varepsilon >0$.
 Denote the function reconstructed by  Algorithm \ref{alg1} from the noisy phaseless samples \eqref{samples1} and \eqref{samples2} by
\begin{equation}\label{rf}
f_\epsilon(t)=\sum_{k=K_1}^{K_2}c_{\epsilon }(k)\psi(t-2k).
\end{equation}
Shown in Figure \ref{fig.} is the performance of reconstructing a function from its noisy phaseless samples in  \eqref{samples1}-\eqref{samples2} by Algorithm \ref{alg1}.

\begin{figure}[h]

\includegraphics[width=6cm]{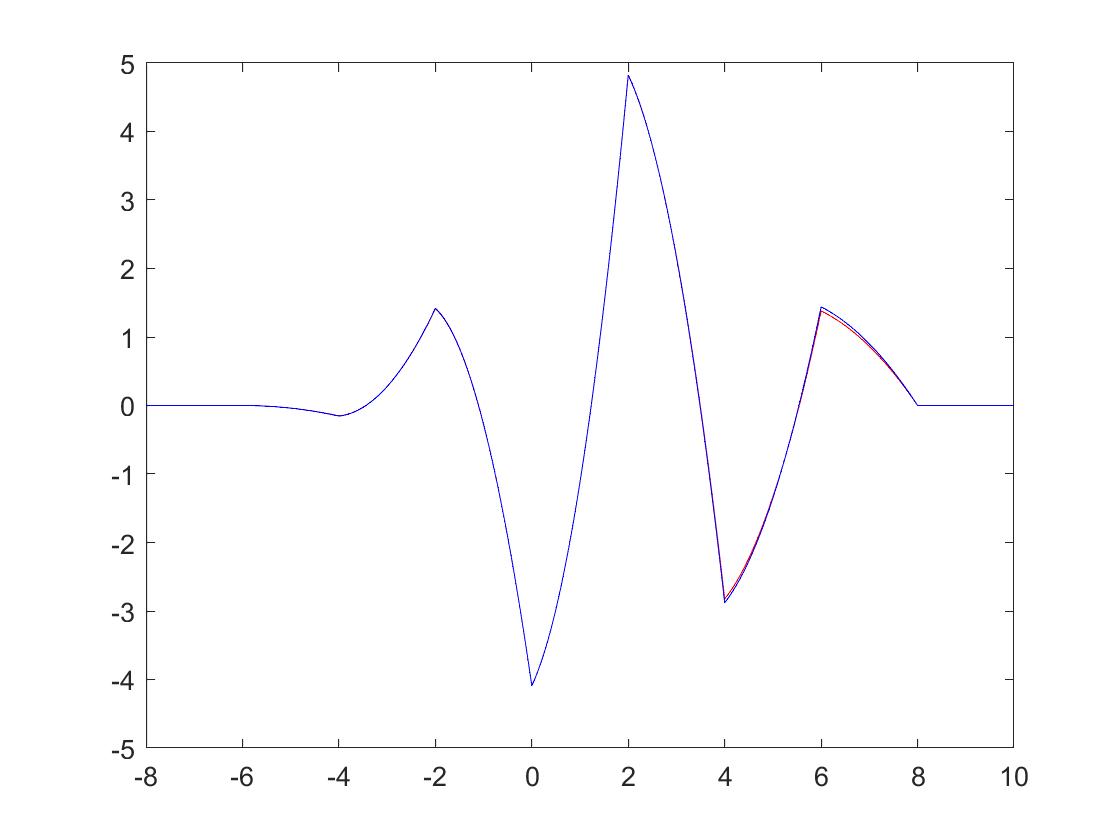} \label{1}
\includegraphics[width=6cm]{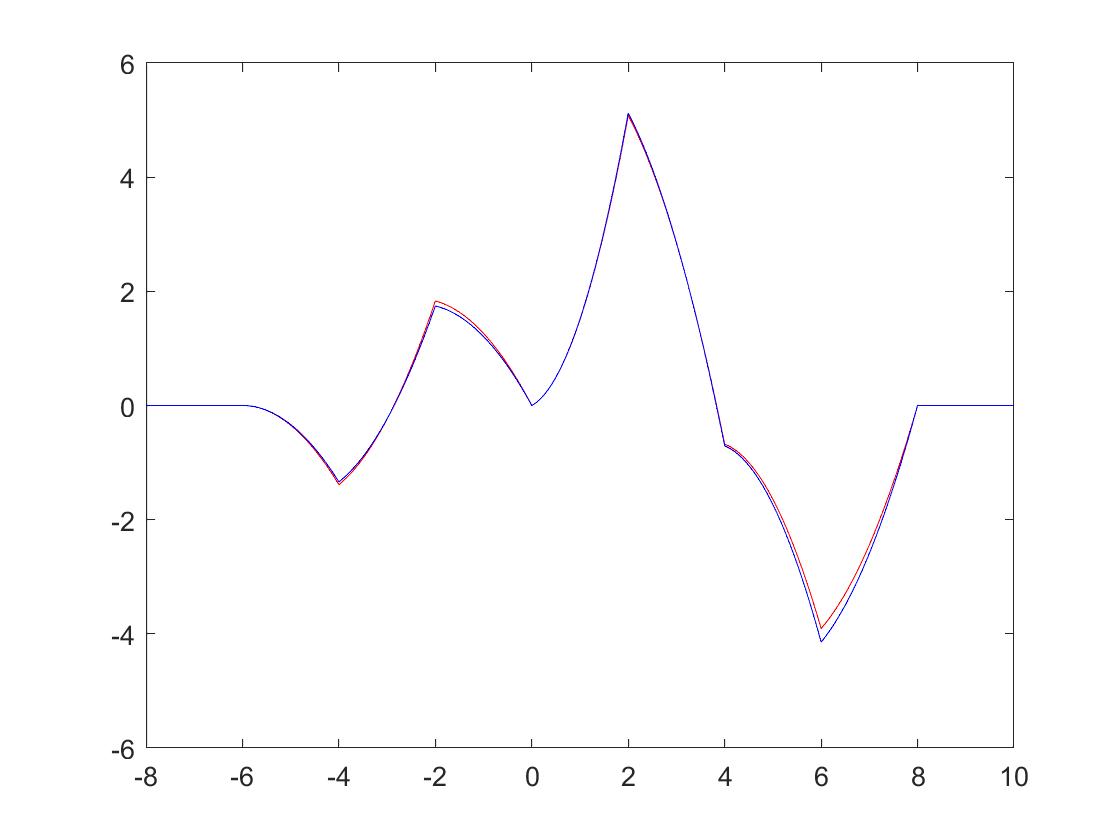} \label{2}

\caption{Plotted on the left is the real parts $\Re f$ (in blue) of the original signal $f$ and $\Re f_\epsilon$ (in red) of the constructed signal $f_\epsilon$, on the right is the imaginary parts  $\Im f$ (in blue) of the original signal $f$ and $\Im f_\epsilon$ (in red) of the constructed signal $f_\epsilon$ via Algorithm \ref{alg1}, and $\max_{k}|\Re (c_k-c_\epsilon(k))/\Re c_k|=0.0429$ and $\max_{k}|\Im (c_k-c_\epsilon(k))/\Im c_k|=0.0611$, where $\varepsilon=10^{-5}$, $K_1=-3$, $K_2=1$ and $\Re c_k,\Im c_k\in[-1,1], -3\le k\le 1$ are randomly selected.   }
\label{fig.}
\end{figure}

\bibliographystyle{plain}

\end{document}